\documentclass[12pt]{amsart}
\usepackage[left=1in,right=1in,top=1in,bottom=1in]{geometry}
\usepackage{latexsym,amsmath,enumerate,amssymb,epsf,graphicx, comment,appendix,amsthm}
\usepackage{mathrsfs}
\usepackage{tikz}
\usetikzlibrary{arrows.meta}
\usepackage{todonotes}
\usepackage{cite}

\usepackage{mathtools}
\DeclarePairedDelimiter{\ceil}{\lceil}{\rceil}

\newcommand{\N}{\mathbb{N}}

\newcommand{\R}{\mathbb{R}}

\DeclareRobustCommand{\rchi}{{\mathpalette\irchi\relax}}
\newcommand{\irchi}[2]{\raisebox{\depth}{$#1\chi$}}

\newcommand{\Ric}{\operatorname{Ric}}

\theoremstyle{plain}\newtheorem{theorem}{Theorem}[section]
\theoremstyle{plain}\newtheorem{lemma}[theorem]{Lemma}
\theoremstyle{plain}
\theoremstyle{plain}\newtheorem{corollary}[theorem]{Corollary}
\theoremstyle{plain}\newtheorem{conjecture}[theorem]{Conjecture}

\theoremstyle{remark}
\newtheorem{remark}[theorem]{Remark}

\theoremstyle{definition}
\newtheorem{definition}[theorem]{Definition}

\numberwithin{equation}{section}

\title[Condensed Ricci Curvature]{Condensed Ricci Curvature of Complete and Strongly Regular Graphs}

\author[Bonini]{Vincent Bonini}

\author[Carroll]{Conor Carroll}

\author[Dinh]{Uyen Dinh}

\author[Dye]{Sydney Dye}

\author[Frederick]{\\Joshua Frederick}

\author[Pearse]{Erin Pearse}

\begin{document}

\maketitle


\begin{abstract}
We study a modified notion of Ollivier's coarse Ricci curvature on graphs introduced by Lin, Lu, and Yau in \cite{LLY1}. We establish a rigidity theorem for complete graphs that shows a connected finite simple graph is complete if and only if the Ricci curvature is strictly greater than one. We then derive explicit Ricci curvature formulas for strongly regular graphs in terms of the graph parameters and the size of a maximal matching in the core neighborhood. As a consequence we are able to derive exact Ricci curvature formulas for strongly regular graphs of girth 4 and 5 using elementary means. An example is provided that shows there is no exact formula for the Ricci curvature for strongly regular graphs of girth $3$ that is purely in terms of graph parameters.
\end{abstract}


\section{Introduction}\label{Sec:Intro}
Lott, Villani \cite{LottVillani} and Sturm \cite{Sturm} discovered a relationship between optimal transport and Ricci curvature on smooth Riemannian manifolds and they pursued a construction of a synthetic notion of Ricci curvature that could be defined independently of differentiable structures. Ollivier \cite{Ollivier,Ollivier2} later introduced a notion of coarse Ricci curvature for Markov chains on metric spaces which has a particularly accessible formulation on graphs. Much work has been done on coarse Ricci curvature on graphs (cf. \cite{BauerJostLiu,LLY1,LLY2,JostLiu,Smith,BhatMuk,BCLMP,Radek}).

Ollivier's coarse Ricci curvature on graphs is defined in terms of transport distance of probability measures (see Section \ref{Sec:TransportDist}). We study a modified notion of Ollivier's coarse Ricci curvature on graphs introduced by Lin, Lu, and Yau in \cite{LLY1}. In this paper we refer to the modified Ricci curvature of \cite{LLY1} as the condensed Ricci curvature. Our first result concerning the condensed Ricci curvature exploits the relationship between Ricci curvature and the eigenvalues of the spectral graph Laplacian to establish a rigidity theorem for complete graphs. This result is stated without proof in \cite[Ex. 1]{LLY1}.

\begin{theorem}\label{Thm:CompleteGraphsIntro}
A connected finite simple graph $G = (V, E)$ is complete if and only if the condensed Ricci curvature $\Bbbk(x,y) > 1$ for all edges $xy \in E$. In this case, $\Bbbk(x,y) = \frac{n}{n-1}$ for all vertices $x,y \in V$.
\end{theorem}

We then turn our attention to the derivation of exact formulas for the condensed Ricci curvature on strongly regular graphs. Explicit formulas and curvature bounds for various forms of Ricci curvature have been established for large classes of graphs (cf. \cite{LLY1,LLY2,JostLiu,Smith,BhatMuk,bakryemery}). In particular, Bakry-{\'E}mery Ricci curvature functions are studied in \cite{bakryemery} and explicit formulas for the Bakry-{\'E}mery Ricci curvature are derived for Cayley graphs and strongly regular graphs of girths $4$ and $5$. This work served as our inspiration to establish explicit formulas for the condensed Ricci curvature on strongly regular graphs using elementary means.

Let $G = (V, E)$ be a strongly regular graph with parameters $(n, d, \alpha, \beta)$. Here $n=|V|$ is the number of vertices, $d$ is the uniform degree of the vertices, $\alpha \geq 0$ is the number of common neighbors for adjacent vertices, and $\beta \geq 1$ is the number of common neighbors for nonadjacent vertices. Given an edge $xy \in E$, let $N_x$ denote the set of vertices that are adjacent to $x$, not including $y$ or any vertices that are adjacent to $y$. Similarly, let $N_y$ denote the set of vertices that are adjacent to $y$, not including $x$ or any vertices that are adjacent to $x$ (see Section \ref{Sec:TransportDist} for more details). We compute the following curvature formulas solely from properties of maximum matchings and the regularity properties of strongly regular graphs.

\begin{theorem}\label{Thm:SRGmatchingIntro}
Let $G = (V, E)$ be strongly regular graph with parameters $(n, d, \alpha, \beta)$. Suppose $xy \in E$ with maximum matching $\mathcal{M}$ of size $m$ between $N_x$ and $N_y$. Then the condensed Ricci curvature
\begin{equation} \label{Eqn:CurvFormulaIntro}
\Bbbk (x, y) = \frac{\alpha + 2}{d} - \frac{|N_x| - m}{d}.
\end{equation}
\end{theorem}

\begin{remark}
We would like to point out that Theorem \ref{Thm:SRGmatchingIntro} holds more generally for regular graphs of diameter $2$. In this case one takes $\alpha$ to be the number of neighbors common to both $x$ and $y$. For simplicity in our presentation, we focus on strongly regular graphs.
\end{remark}

The explicit formula $\Bbbk (x, y) = \frac{\alpha + 2}d$ was previously established in \cite[Theorem 6.3]{Smith} for general graphs with perfect matchings between $N_x$ and $N_y$. In the setting of strongly regular graphs Theorem~\ref{Thm:SRGmatchingIntro} is thus an extension of the results of \cite{Smith} to the case of maximum matchings. These results give insight into the importance of understanding matchings between the neighbor sets $N_x$ and $N_y$ in the derivation of exact Ricci curvature formulas. In particular, we are able to exploit this idea to derive the explicit formulas for the Ricci curvature of strongly regular graphs of girths $4$ and $5$ using elementary means.

For strongly regular graphs $G = (V, E)$ of girth $5$, it follows that adjacent vertices share $\alpha = 0$ common neighbors. Moreover, for a given edge $xy \in E$, there are no edges between the neighbor sets $N_x$ and $N_y$. Then, as a simple consequence of Theorem \ref{Thm:SRGmatchingIntro}, we have the following Ricci curvature formula.

\begin{theorem} \label{Thm:SRGgirth5intro}
Let $G = (V, E)$ be strongly regular graph with parameters $(n, d, \alpha, \beta)$. If the girth of $G$ is $5$, then condensed Ricci curvature
$$\Bbbk(x ,y) = \frac{3}{d} - 1$$ 
for all $xy \in E$.
\end{theorem}

From Theorem \ref{Thm:SRGgirth5intro} it is easy to see that the condensed Ricci curvature along any edge of the $5$-cycle $C_5$ is $\Bbbk = \frac{1}{2}$ while the condensed Ricci curvature along any edge of the Peterson graph is $\Bbbk = 0$. The fact that these are the only nonnegatively curved strongly regular graphs of girth $5$ then follows from Theorem \ref{Thm:SRGgirth5intro} and the classification of Moore graphs of diameter $2$ due to Hoffman and Singleton \cite{HS}.

\begin{corollary} \label{Cor:SRGgirth5NegativeCurvIntro}
The $5$-cycle and the Petersen graph are the only strongly regular graphs of girth $5$ with nonnegative condensed Ricci curvature along edges. 
\end{corollary}

For strongly regular graphs $G = (V, E)$ of girth $4$, it follows that adjacent vertices share $\alpha = 0$ common neighbors. For a given edge $xy \in E$, we appeal to Hall's theorem and the pigeonhole principle to establish a perfect matching between the neighbor sets $N_x$ and $N_y$. Then, as a consequence of Theorem \ref{Thm:SRGmatchingIntro}, we have the following Ricci curvature formula.

\begin{theorem} \label{Thm:SRGgirth4intro}
Let $G = (V, E)$ be strongly regular graph with parameters $(n, d, \alpha, \beta)$. If the girth of $G$ is $4$, then condensed Ricci curvature
$$\Bbbk(x ,y) = \frac{2}{d}$$
for all $xy \in E$. In particular, all strongly regular graphs of girth $4$ have positive condensed Ricci curvature.
\end{theorem}

For strongly regular graphs with girth $3$, it turns out that there does not exist an exact formula for the condensed Ricci curvature solely in terms of the graph parameters. This is due to the fact that there are nonisomorphic girth $3$ strongly regular graphs with the same graph parameters but different Ricci curvatures. For example, the $4 \times 4$ Rook's graph and the Shrikhande graph both have parameters $(16,6,2,2)$ but their condensed Ricci curvatures are $\Bbbk=\frac23$ and $\Bbbk=\frac13$, respectively. This is due to the fact that for a given edge $xy$ there is a perfect matching between neighbor sets $N_x$ and $N_y$ for the Rook's graph but the maximum matching between $N_x$ and $N_y$ is size $m=1$ for the Shrikhande graph.

On the other hand, strongly regular conference graphs exhibit many symmetry properties that we believe always lead to perfect matchings between neighbor sets $N_x$ and $N_y$. We plan to address conference graphs in a forthcoming paper. For now we make the following conjecture for conference graphs, which can be realized as a direct consequence of Theorem \ref{Thm:SRGmatchingIntro} provided there are always perfect matchings between $N_x$ and $N_y$.

\begin{conjecture}\label{Conj:ConferenceGraphsIntro}
Let $G=(V, E)$ be a strongly regular conference graph with parameters\\ $(4\beta + 1, 2\beta, \beta -1, \beta)$ with $\beta \geq 2$. Then the condensed Ricci curvature
$$\Bbbk(x, y) = \frac12 + \frac{1}{2\beta}$$
for all $xy \in E$. In particular, all strongly regular conference graphs have positive condensed Ricci curvature strictly greater than $\frac{1}{2}$.
\end{conjecture}

This paper is organized as follows: In Section \ref{Sec:TransportDist} we formally define transport distance between probability measures on graphs and the condensed Ricci curvature. We  briefly discuss our choice of notation and terminology and we state some key results from \cite{BCLMP,BhatMuk} that simplify the computation of condensed Ricci curvature. In Section \ref{Sec:CompleteGraphs} we introduce the spectral graph Laplacian and prove Theorem \ref{Thm:CompleteGraphsIntro}. In Section \ref{Sec:SRGraphs} we focus on strongly regular graphs and establish our exact curvature formulas.


\section*{Acknowledgements}
The authors were partially supported by NSF FURST grant DMS-$1620552$, the Cal Poly Frost Fund, CSU Fresno, and Cal Poly, San Luis Obispo. The authors would like to thank CSU Fresno for their hospitality during their summer participation in the FURST program.


\section{Transport Distance and Condensed Ricci Curvature on Graphs}\label{Sec:TransportDist}
To motivate Ollivier's formulation of coarse Ricci curvature we turn to the setting of smooth Riemannian manifolds. Suppose that $x_1$ and $x_2$ are points on an $N$-dimensional Riemannian manifold $M$ and that $m_i$ is the uniform measure on the respective geodesic ball $B(x_i,\varepsilon)$ of radius $\varepsilon$ centered at $x_i$. Let $\delta = d(x_1,x_2)$ denote the geodesic distance between $x_1$ and $x_2$. If $B(x_1,\varepsilon)$ is parallel transported to $B(x_2,\varepsilon)$ along a geodesic from $x_1$ to $x_2$ in the direction of tangent vector $v$ at $x_1$, then the transport distance $W(m_1,m_2)$ between the measures and the Ricci curvature on $M$ satisfy the relation \cite{Ollivier}
\begin{align*}
W(m_1,m_2) = \delta \left(1-\frac{\varepsilon^2}{2(N+2)} \Ric(v, v) + O(\varepsilon^3 + \varepsilon^2 \delta)\right) \quad \text{as} \ \varepsilon, \delta \to 0.
\end{align*}

As noted in \cite{Ollivier2}, this relation allows one to interpret Ricci curvature on Riemannian manifolds as a measure of the average distance traveled between points of geodesic balls versus their centers under parallel transport. Using this relationship to guide intuition, Ollivier defined the Ricci curvature of Markov chains on metric spaces
\begin{align*}
\Ric(x_1,x_2)  = 1 - \frac{W(m_1,m_2)}{d(x_1,x_2)}
\end{align*}
where $W(m_1,m_2)$ denotes the general transport or Wasserstein distance between measures $m_i$ based at $x_i$ and $d(x_1,x_2)$ is the metric distance between $x_1$ and $x_2$. This definition provides a synthetic notion of Ricci curvature of Markov chains on metric spaces and its geometric validation comes in the form of analogues of the classical theorems of Bonnet-Myers, Lichnerowicz, and Gromov-L\'evy \cite{Ollivier,Ollivier2}.

Ollivier's coarse Ricci curvature on graphs is defined in terms of transport distance of probability measures and standard graph distance \cite{Ollivier,Ollivier2}. Let $G=(V,E)$ be a nontrivial, locally finite, undirected, connected, simple graph with standard graph shortest path distance function
$$
\rho:V \times V \to  {\N} \cup \{0\}.
$$
A {\it probability measure} or {\it probability mass distribution} on $G=(V,E)$ is a real valued function $\mu : V \to [0,1]$ such that 
$$
\sum_{v \in V}\mu(v) = 1.
$$
Given probability measures $\mu$ and $\nu$ on $G$, a {\it coupling} or {\it transport plan} between $\mu$ and $\nu$ is a probability measure $\xi : V \times V \to [0, 1]$ such that
$$
\sum_{w \in V} \xi(v, w) = \mu \quad \text{and} \quad  \sum_{v \in V} \xi(v, w) = \nu.
$$
The terminology transport plan has economic origins and comes from the optimal transport problem where one interprets a coupling as a plan for transporting mass or goods from one distribution center to another.

The $L^1$-Wasserstein or transport distance is a metric on the space of probability measures on $G$ that quantifies the optimal {\it transport cost} between probability measures on $G$.

\begin{definition} \label{Def:TransportMetric}
The {\it $L^1$-Wasserstein} or {\it transport distance} between probability measures $\mu$ and $\nu$ on a graph $G=(V,E)$ is given by
$$
W(\mu,\nu) = \inf_{\xi \in {\rchi} (\mu,\nu)} \sum_{x \in V} \sum_{y \in V} \rho(x,y) \xi (x,y)
$$
where $ {\rchi}(\mu, \nu)$ is the set of all transport plans between $\mu$ and $\nu$.
\end{definition}

We consider a modified notion of Ollivier's coarse Ricci curvature on graphs introduced by Lin, Lu, and Yau in \cite{LLY1} that is based on probability measures of the form
\begin{equation} \label{Eqn:EpsilonMeasure}
m_x^{\varepsilon}(y) = 
\begin{cases} 
1 -\varepsilon & y=x,\\
\frac\varepsilon{\deg(x)} &y \in \Gamma(x), \\
0 & \text{otherwise}.
\end{cases}
\end{equation}
For $x,y \in V$ the $\varepsilon$-Ollivier Ricci curvature is then defined for $0 \leq \varepsilon \leq 1$ by
$$ 
\kappa_{\varepsilon}(x, y) = 1 - \frac{W(m_x^{\varepsilon}, m_y^{\varepsilon})}{\rho(x,y)}
$$
and the modified Ollivier Ricci curvature of is then defined for $x,y \in V$ by
$$
\Bbbk(x, y) = \lim_{\varepsilon \to 0} \frac{\kappa_{\varepsilon}(x,y)}{\varepsilon}.
$$

We refer to the modified Ricci curvature as the condensed Ricci curvature. Our definitions differ slightly from those of \cite{LLY1} in that we use $1-\varepsilon$ in place of $\alpha$ as in \cite{Smith,Radek}. The reason for this choice in notation and terminology is that we interpret the probability measure \eqref{Eqn:EpsilonMeasure} as encoding an $\varepsilon$-active random walk that becomes stationary as $\varepsilon \to 0$. One can view this phenomena as a concentration or condensation of measure. This formulation also more closely resembles the notion of coarse Ricci curvature of continuous-time Markov chains given by the derivative
$$\kappa(x,y) = - \frac{d}{dt} \frac{W(m_x^{t}, m_y^{t})}{\rho(x,y)}.$$

Computing Ricci curvature on graphs requires solving the optimal transport problem. By definition a given transport plan between measures $m_x^{\varepsilon}$ and $m_x^{\varepsilon}$ gives an upper bound for the transport distance $W(m_x^{\varepsilon}, m_y^{\varepsilon})$ and therefore a lower bound for condensed Ricci curvature. While an optimal transport plan may be intuitive, realizing that it achieves the transport distance directly from the definition is technically challenging as it requires solving a linear programming problem. Fortunately, there is a dual formulation of this optimization problem in terms of $1$-Lipschitz functions that one can use to provide lower bounds to transport distance and consequently upper bounds for condensed Ricci curvature.

\begin{theorem}[Kantorovich Duality Theorem] \label{Thm:Kantorovich}
The $L^1$-Wasserstein or transport distance between probability measures $\mu$ and $\nu$ on a graph $G=(V,E)$ is given by
$$
W(\mu,\nu) = \sup_{f \in Lip(1)} \sum_{z \in V} f(z) (\mu(z)-\nu(z))
$$
where 
$$Lip(1) = \{f : V \to \R \ \big| \ \lvert f(x) - f(y) \rvert \leq \rho(x, y) \ \text{for all} \ x, y \in V \}$$
is the space of $1$-Lipschitz functions on $G$.
\end{theorem}

Due to the Kantorovich Duality Theorem we refer to $1$-Lipschitz functions on a graph $G=(V,E)$ as Kantorovich potential functions. With the Kantorovich Duality Theorem in hand one can compute exact formulas for Ricci curvature by specifying a transport plan and a Kantorovich potential that provide bounds that pinch the condensed Ricci curvature and force equality. Another useful fact for computing condensed Ricci curvature due to \cite{BCLMP} is that the Ricci activeness function $\varepsilon \mapsto \Bbbk_{\varepsilon}$ is piecewise linear with $2$ linear parts on regular graphs of degree $d$ and the $\varepsilon$-Ollivier Ricci curvature differs from the condensed Ricci curvature only by scale for all values $0 \leq \varepsilon \leq \frac{d}{d+1}$. More precisely, on a regular graph $G=(V,E)$ of degree $d$, for any $x,y \in V$
$$
\kappa_{\varepsilon}(x, y) = \varepsilon \ \Bbbk (x,y)\quad \text{for} \ 0 \leq \varepsilon \leq \frac{d}{d+1}.
$$
For simplicity, in our computations we take $\varepsilon = \frac{1}{2}$ and compute the condensed Ricci curvature
$$ \label{Def:HalfEps}
\Bbbk(x, y) = 2\kappa_{\frac{1}{2}}(x, y).
$$

In Riemannian geometry, Ricci curvature is a local quantity so in the context of graphs it seems natural to restrict one's attention to computing Ricci curvature along edges. With our conventions, this restriction allows for further simplification in computing Ricci curvature due to \cite{BhatMuk}. For vertices $x \in V$ let
$$
\Gamma(x) = \{y \in V \ \big| \ \rho(x,y)=1 \}  = \{y \in V \ \big| \ xy \in E  \}
$$
denote the neighbor set of $x$.  Given an edge $xy \in E$ we denote the set of common neighbors or {\it triangle set} of $x$ and $y$ by
$$\nabla_{xy} = \Gamma(x) \cap \Gamma(y).$$
Then we further decompose the $1$-step neighborhoods of $x$ and $y$ into disjoint unions
$$
\Gamma(x) = N_x \cup \nabla_{xy} \cup \{y\} \quad \text{and} \quad \Gamma(y) = N_y \cup \nabla_{xy} \cup \{x\}
$$
where
$$
N_x = \Gamma(x) \setminus (\nabla_{xy} \cup \{y\}) \quad \text{and} \quad N_y = \Gamma(y) \setminus (\nabla_{xy} \cup \{x\}).
$$
We denote the set of common $2$-step neighbors or {\it pentagon set} of $x$ and $y$ by
$$
P_{xy} = \{ z \in V \ \big| \ \rho(x, z) = 2 \ \text{ and } \ \rho(y, z) = 2\}.
$$
As in \cite{BhatMuk}, we refer to the disjoint union of vertices
\begin{equation} \label{Eqn:CoreNbrhd}
\mathcal{N}_{xy} =  \{x\} \cup \{y\} \cup \nabla_{xy} \cup N_x \cup N_y \cup P_{xy}
\end{equation}
as the \emph{core neighborhood} of an edge $xy \in E$. It is immediate that the core neighborhood represents a partition of all vertices with distance less than or equal to $2$ from $x$ and $y$. Moreover, the Reduction Lemma \cite[Lemma 2.3]{BhatMuk} shows that when computing the condensed Ricci curvature of an edge $xy \in E$, it is sufficient to consider only the induced subgraph of $G$ lying within the core neighborhood $\mathcal{N}_{xy}$.


\section{A Rigidity Theorem for Complete Graphs}\label{Sec:CompleteGraphs}
In this section we establish a rigidity theorem for complete graphs that shows that a complete graph is the only graph with condensed Ricci curvature strictly greater than one. Given a connected finite simple graph $G=(V,E)$ with $n$ vertices $v_1,\dots,v_n$, the adjacency matrix $A$ of $G$ is a symmetric $n \times n$ matrix with entries 
$$
\begin{cases}
a_{ij} =1 &\text{if} \ v_iv_j \in E\\
a_{ij} =0 &\text{otherwise}.
\end{cases}
$$
Let $d_i$ denote the degree of vertex $v_i$ and let $D = \text{diag}(d_1,\dots,d_n)$
be the $n \times n$ diagonal matrix of degrees. The spectral graph Laplacian is then defined
$$
L = I - D^{-\frac{1}{2}}AD^{-\frac{1}{2}}.
$$
With these conventions it follows that the eigenvalues of the spectral graph Laplacian are nondecreasing with
$$
0=\lambda_0 \leq \lambda_1 \leq \cdots \leq \lambda_{n-1}
$$
where the multiplicity of the $0$ eigenvalue is the number of connected components of $G$. Hence, a finite graph is connected if and only if $\lambda_1 > \lambda_0 = 0$. Similar to the spectrum of the metric Laplacian on Riemannian manifolds, the spectrum of the spectral Laplacian on graphs encodes geometric and structural properties of a graph. 

In the context of Riemannian geometry the Lichnerowicz theorem \cite{Lich} is a widely celebrated theorem that provides a lower bound for the first nonzero eigenvalue of the metric Laplacian when the Ricci curvature has a strict positive lower bound. More precisely, if $(M^n,g)$ is an $n$-dimensional compact Riemannian manifold with
$$Ric_g \geq (n-1) k g$$
for some positive constant $k>0$, then the first positive eigenvalue of the metric Laplacian $\lambda_1 \geq nk$. The graph analog of the Lichnerowicz theorem similarly shows that a positive lower bound on condensed Ricci curvature serves as a lower bound on the first nonzero eigenvalue of the spectral graph Laplacian.

\begin{theorem} \label{Thm:GraphLich} \textup{\cite[Theorem 4.2]{LLY1}} 
Suppose $G=(V,E)$ is a connected finite simple graph with condensed Ricci curvature
$$\Bbbk(x,y) \geq \kappa_0 > 0$$ 
for some positive constant $\kappa_0$ and all edges $xy \in E$. Then the first nonzero eigenvalue of the spectral graph Laplacian $\lambda_1 \geq \kappa_0$.
\end{theorem}

In order to establish a rigidity theorem for complete graphs we first record some well known results on the first nonzero eigenvalue of the spectral graph Laplacian.

\begin{theorem} \textup{\label{Thm:Lambda1} \cite{Chung}}
Let $G=(V,E)$ be a simple graph on $n$ vertices. Then 
$$\lambda_1 \leq \frac{n}{n-1}$$ 
with equality if and only if $G$ is complete. Moreover, complete graphs are the only graphs with $\lambda_1 > 1$.
\end{theorem}

We now prove the following rigidity theorem through direct computation and the aid of Theorems \ref{Thm:GraphLich} and \ref{Thm:Lambda1}.

\begin{theorem} \label{Thm:CompleteGraphs}
A connected finite simple graph $G = (V, E)$ is complete if and only if the condensed Ricci curvature $\Bbbk(x,y) > 1$ for all edges $xy \in E$.
\end{theorem}

\begin{proof}
Let $G=(V,E)$ be the complete graph on $n$ vertices and let $xy \in E$. Then the triangle set $\nabla_{xy} = V \setminus \{x, y\}$
since $G$ is complete. Let $m_x$ and $m_y$ denote $\varepsilon=\frac{1}{2}$ probability measures based at $x$ and $y$, respectively, defined as in \eqref{Eqn:EpsilonMeasure}. Consider the transport plan $\xi:V \times V \to [0,1]$ between $m_x$ and $m_y$ given by 
$$\xi(v, w) = 
\begin{cases}
\frac{1}{2} - \frac{1}{2(n-1)} & v = x, w = y \\
\frac{1}{2(n-1)} & v = w \\
0 & \text{otherwise},
\end{cases}$$
which transports excess mass directly from $x$ to $y$ and leaves the remainder of the distribution fixed. Then by Definition \ref{Def:TransportMetric} the transport distance
$$W(m_x ,m_y) \leq \sum_{v\in V}\sum_{w\in V} \rho(v,w)\xi(v,w) = \xi(x,y) = \frac{1}{2} - \frac{1}{2(n-1)}$$
and therefore the condensed Ricci curvature
\begin{equation} \label{Eqn:CompleteRicciLB}
\Bbbk(x, y) = 2\kappa_{\frac{1}{2}}(x, y) \geq 2\left(1 - \left(\frac{1}{2} - \frac{1}{2(n - 1)}\right)\right) = 1 + \frac{1}{n-1} > 1.
\end{equation}

Conversely, let $G=(V,E)$ be a connected finite simple graph with condensed Ricci curvature $\Bbbk(x ,y) > 1$ for all edges $xy \in E$. Then by the graph Lichnerowicz Theorem \ref{Thm:GraphLich} of \cite{LLY1}, it follows $\lambda_1 > 1$ and therefore $G$ is complete by Theorem \ref{Thm:Lambda1}. \end{proof}

\begin{remark}\textup{
Using the $1$-Lipschitz function $f:V\to{\R}$ defined by
$$f(v) = 
\begin{cases}
1 & v = x\\
0 & \text{otherwise},
\end{cases}$$
it follows from the Kantorovich Duality Theorem \ref{Thm:Kantorovich} that the transport distance
$$W(m_x ,m_y) \geq \sum_{v \in V}f(v)(m_x(v) - m_y(v)) = m_x(x) - m_y(x) = \frac{1}{2} - \frac{1}{2(n-1)}$$
and therefore the condensed Ricci curvature
\begin{equation} \label{Eqn:CompleteRicciUB}
\Bbbk(x, y) = 2\kappa_{\frac{1}{2}}(x, y) \leq 2\left(1 - \left(\frac{1}{2} - \frac{1}{2(n - 1)}\right)\right) = 1 + \frac{1}{n-1}.
\end{equation} 
Hence, from \eqref{Eqn:CompleteRicciLB} and \eqref{Eqn:CompleteRicciUB} we see that the condensed Ricci of the complete graph on $n$ vertices is
$$\Bbbk(x, y) = 1+\frac{1}{n-1} = \frac{n}{n-1}$$
for all edges $xy \in E$ as stated in \cite[Ex. 1]{LLY1}.
}\end{remark}


\section{Strongly Regular Graphs}\label{Sec:SRGraphs}
In this section we derive an explicit formula for the condensed Ricci curvature of strongly regular graphs $G=(V,E)$ in terms of the graph parameters and the size of a maximum matching in the core neighborhood. Recall that a simple, undirected, finite graph $G=(V,E)$ is said to be a strongly regular graph with parameters $(n, d, \alpha, \beta)$ if $G$ has $n$ vertices of constant degree $d$ where any two adjacent vertices share $\alpha \geq 0$ common neighbors and any two nonadjacent vertices share $\beta \geq 1$ common neighbors. Given an edge $xy \in E$, we consider the core neighborhood decomposition \eqref{Eqn:CoreNbrhd} given by
$$\mathcal{N}_{xy} =  \{x\} \cup \{y\} \cup \nabla_{xy} \cup N_x \cup N_y \cup P_{xy}$$
that was introduced in \cite{BhatMuk} and where the neighbor sets
$$N_x = \Gamma(x) \setminus (\nabla_{xy} \cup \{y\}) \quad \text{and} \quad N_y = \Gamma(x) \setminus (\nabla_{xy} \cup \{y\}).$$
Since strongly regular graphs have diameter $2$, it follows that $V=\mathcal{N}_{xy}$. Hence, for strongly regular graphs the induced subgraph lying within the core neighborhood of any edge accounts for the entire graph $G$. Moreover, from strong regularity we see $|\nabla_{xy}| = \alpha$ so that 
$$|N_x| = |N_y| = d - \alpha -1 \quad \text{and} \quad |P_{xy}| = n - 2d - \alpha.$$

Our strategy in the derivation of an explicit formula for the condensed Ricci curvature of strongly regular graphs reformulates the problem of finding candidates for an optimal transport plan and an optimal Kantorovich potential into a maximum matching problem.

\begin{definition} \label{Def:Matching}
A {\it matching} on a finite undirected simple graph $G=(V,E)$ is a subset of edges $\mathcal{M} \subseteq E$ such that no vertex in $V$ is incident to more than one edge in $\mathcal{M}$. The size of a matching $\mathcal{M}$ on $G$ is the total number of edges in $\mathcal{M}$. A matching $\mathcal{M}$ on $G$ is said to be a {\it maximum matching} if it contains the largest possible number of edges. A maximum matching between disjoint subsets of vertices $S,T \subset V$ is said to be a {\it perfect matching} if every vertex in $S$ and $T$ is incident to exactly one edge of the matching. 
\end{definition}

Let $G = (V, E)$ be strongly regular graph with edge $xy \in E$. Given a maximum matching $\mathcal{M}$ between the neighbor sets $N_x$ and $N_y$, we denote the vertices that are matched in $N_x$ and $N_y$ by $M_x$ and $M_y$, respectively. If $\mathcal{M}$ is not a perfect matching we denote the remaining unmatched vertices in $N_x$ and $N_y$ by $U_x = N_x \setminus M_x$ and $U_y = N_y \setminus M_y$, respectively. With this decomposition of $N_x$ and $N_y$ into matched and unmatched vertices it is fairly straightforward to develop a candidate for an optimal transport plan. However, developing a candidate for an optimal Kantorovich potential requires a deeper understanding of the incidence relations between edges in a given maximum matching and the remaining edges between $N_x$ and $N_y$.

\begin{definition} \label{Def:Paths}
Let $\mathcal{M}$ be a matching on an undirected simple graph $G=(V,E)$. An {\it alternating path} $P$ in $G$ is a path $P=v_0v_1 \cdots v_k$ such that the initial vertex $v_0$ is unmatched and $v_i v_{i + 1} \in \mathcal{M}$ if and only if $v_{i - 1} v_{i} \not \in \mathcal{M}$ for all $0 < i < k$. An {\it augmenting path} $P$ in $G$ is an alternating path $P=v_0v_1 \cdots v_k$ such that the terminal vertex $v_k$ is also unmatched.
\end{definition}
                
We refer to a single vertex as a trivial path. Hence, with this convention any unmatched vertex is by itself a trivial augmenting path. One of the main tools in our construction of a candidate for an optimal Kantorovich potential is Berge's Lemma \cite{Berge}.

\begin{lemma}[Berge's Lemma] \label{Lem:Berge}
A matching $\mathcal{M}$ on an undirected simple graph $G=(V,E)$ is a maximum matching if and only if $G$ contains no nontrivial augmenting paths.
\end{lemma}                
                
Given a matching $\mathcal{M}$ on a graph $G=(V,E)$ and a subset of vertices $U \subseteq V$, we denote the subset of matched vertices in $U$ by $\mathcal{M}(U)$.  For a subset of vertices $S \subseteq V$ we denote the subset of alternating paths with initial vertex in $S$ by $A_S$. For subsets of vertices $S,T \subseteq V$ we denote the set of vertices in $T$ that lie along an alternating path initiated in $S$ by $A_S(T)$. We use the contrapositive of the following lemma to show that certain vertices in the neighbor sets $N_x$ and $N_y$ have no edges between them. Ultimately, this allows us to ensure our candidate for an optimal Kantorovich potential is $1$-Lipschitz.
  
\begin{lemma} \label{Lem:Lip}
Let $H = (V, E)$ be a finite, undirected, bipartite graph with parts $S, T$. Let $\mathcal{M}$ be a maximum matching on $H$ and suppose $v \in A_{S}(S)$ and $w \in T$. If $vw \in E$, then $w \in A_{S}(T)$.
\end{lemma}
            
\begin{proof}
First note that since $H$ is bipartite, it follows that an alternating path initiated in $S$ can only terminate in an unmatched vertex in $T$. Therefore, if $v \in A_{S}(S)$ is an unmatched vertex, then $v$ must be the initial vertex of an alternating path. But then $w$ must be a matched vertex so that $w \in A_{S}(T)$ since otherwise $vw \in E$ is an augmenting path, contradicting Berge's Lemma as $\mathcal{M}$ is a maximum matching. Now assume that $v$ is a matched vertex and that $P = v_0v_1 \dots v_iv$ is an alternating path initiated at $v_0 \neq v \in S$ that terminates at $v \in S$. Then since $v_0 \in S$ is an unmatched vertex and $H$ is bipartite, it follows that $v_i v \in \mathcal{M}$ since $P$ is an alternating path. But then $vw \not\in \mathcal{M}$ since otherwise $v$ is incident to more than one edge in $\mathcal{M}$. But then the concatenation $P' =  v_0v_1 \dots v_ivw$ is an alternating path so that $w \in A_{S}(T)$.
\end{proof}

We use this final lemma to compute a lower bound for condensed Ricci curvature from our Kantorovich potential in terms of the number of vertices in $N_x$ and $N_y$ along alternating paths initiated in $N_y$.
            
\begin{lemma} \label{Lem:CountingPaths}
Let $H = (V, E)$ be a finite, undirected, bipartite graph with parts $S, T$. Suppose $\mathcal{M}$ is a maximum matching of size $m$ on $H$. Then 
$$| A_{S}(S) | = |  A_{S}(T) | + |S| - m.$$
\end{lemma}
                    
\begin{proof}
Let $v \in A_{S}(T)$ and suppose $P=v_0v_1 \cdots v_iv$ is an alternating path initiated at $v_0 \in S$ that terminates at $v \in T$. Then since $v_0 \in S$ is an unmatched vertex and $H$ is bipartite, it follows that $v_iv \not\in \mathcal{M}$ since $P$ is an alternating path. If $v$ is an unmatched vertex, then $P$ is an augmenting path contrary to Berge's Lemma as $\mathcal{M}$ is a maximum matching. Hence, $v$ must be a matched vertex and therefore since matched vertices are incident to a single edge in the matching, it follows that there exits a unique vertex $w \in S$ such that $vw \in \mathcal{M}$. But then the concatenation $P' = v_0v_1\cdots vw$ is an alternating path initiated in $S$ that terminates at the matched vertex $w \in \mathcal{M}(S)$. Hence, for each $v \in A_{S}(T)$ there is a unique $w \in  A_{S}(S) \cap \mathcal{M}(S)$ with $vw \in \mathcal{M}$. 

On the other hand, given $w \in  A_{S}(S) \cap \mathcal{M}(S)$ suppose $P=w_0w_1 \cdots w_iw$ is an alternating path initiated in $S$ that terminates at $w \in S$. Then since $w$ is a matched vertex by assumption, it is clear that $P$ is a nontrivial path and that $w_i \in T$ is the unique matched vertex with $w_iw \in \mathcal{M}$. But then $P'=w_0w_1 \cdots w_i$ is an alternating path initiated in $S$ that terminates at $w_i \in T$ so that $w_i \in A_{S}(T)$. Hence, for each $w \in  A_{S}(S) \cap \mathcal{M}(S)$ there is a unique $v \in A_{S}(T)$ with $vw \in \mathcal{M}$.

Thus, since $A_S(T) \subseteq \mathcal{M}(T)$, we have a bijection between the subsets $A_{S}(T)$ and $A_{S}(S) \cap \mathcal{M}(S)$ and therefore $|A_{S}(T)| = | A_{S}(S) \cap \mathcal{M}(S)|$. Noting $S \setminus \mathcal{M}(S) \subseteq A_S(S)$ since under our conventions all unmatched vertices in $S$ are trivial alternating paths initiated in $S$, it follows  
\begin{align*}
|A_{S}(T)| &= | A_{S}(S) \cap \mathcal{M}(S) | = | A_{S}(S)  \setminus (S \setminus \mathcal{M}(S)) | \\
& = |A_{S}(S)| - | S \setminus \mathcal{M}(S) | = | A_{S}(S) | - (|S| - m).
\end{align*}
Hence,
$$|A_{S}(S)| = |A_{S}(T)| + |S| - m$$
as desired.
\end{proof}
            
We are now in a position to derive an explicit formula for condensed Ricci curvature of an edge $xy \in E$ in terms of the graph parameters and the size of a maximum matching between the neighbor sets $N_x$ and $N_y$. 

\begin{theorem} \label{Thm:SRGmatching}
Let $G = (V, E)$ be strongly regular graph with parameters $(n, d, \alpha, \beta)$. Suppose $xy \in E$ with maximum matching $\mathcal{M}$ of size $m$ between $N_x$ and $N_y$. Then the condensed Ricci curvature
$$\Bbbk (x, y) = \frac{\alpha + 2}{d} - \frac{|N_x| - m}{d}.$$
\end{theorem}

\begin{proof}
Let $ xy \in E $ and consider the core neighborhood decomposition of $G$ as in \eqref{Eqn:CoreNbrhd}. Let $H$ denote the induced bipartite subgraph consisting of all edges in $E$ between vertices in $N_x$ and $N_y$. Suppose that $\mathcal{M}$ is a maximum matching on $H$ of size $|\mathcal{M}| = m$. Assuming $\mathcal{M}$ is not a perfect matching, as above we denote the matched and unmatched vertices in $N_x$ and $N_y$ by $M_x, M_y$ and $U_x, U_y$, respectively. Note that $|U_x| = |U_y|$ since $|N_x| = |N_y|$ and $|M_x|=|M_y|$. For $i=1, \dots, |N_x| - m$, let $x_i^u \in N_x$ and $y_i^u \in N_y$ denote the unmatched vertices in $H$ and for $j=1, \dots, m$, let $x_j^m \in N_x$ and $y_j^m \in N_y$ denote the matched vertices in $H$. Up to reordering of indices we may assume $x_i^my_i^m \in \mathcal{M}$ for $i=1,\dots, m$.

Clearly, there are no edges in $E$ incident to unmatched vertices since $\mathcal{M}$ is a maximum matching so $x_i^uy_j^u \not\in E$ for all $1 \leq i,j \leq |N_x| - m$. But then $\rho(x_i^u,y_i^u) = 2$ in $G$ since $x_i^u$ and $y_i^u$ are not adjacent and $G$ has diameter $2$. Hence, for $1 \leq i \leq |N_x| -m$, we may pair unmatched vertices $x_i^u$ with $y_i^u$ along $2$-step paths in $G$. For simplicity in notation write $x_i^u y_i^u  \in \mathcal{M}^2$ to denote unmatched vertices $x_i^u \in U_x$ and $y_i^u \in U_y$ that have been paired along a $2$-step path. Together these matchings/pairings induce a transport plan that moves excess mass from $x$ to $y$ and that moves mass from $x_i^m \in M_x$,  $x_i^u \in U_x$ directly to their matched/paired neighbors $y_i^m \in M_y$, $y_i^u \in U_y$ along $1$-step and $2$-step paths, respectively. More precisely, consider the transport plan defined by
\begin{equation} \label{Eqn:SRGTransportPlan}
\xi(v,w)=
\begin{cases}
\frac{1}{2} - \frac{1}{2d} & v = x, w = y \\
\frac{1}{2d} & v = w \in \nabla_{xy} \cup \{x\} \cup \{y\}\\
\frac{1}{2d} & vw \in \mathcal{M}\\
\frac{1}{2d} & vw \in \mathcal{M}^2\\
0 & \text{otherwise}.
\end{cases}
\end{equation}
Noting that $|N_x| = |N_y| = d - \alpha - 1$ it follows that the transport distance
\begin{align*}
W(m_x, m_y) \leq \sum_{v \in V} \sum_{w \in V} \rho(v,w)\xi(v,w) &= \left(\frac{1}{2} - \frac{1}{2d}\right) + m \cdot \frac{1}{2d}  + \left(|N_x| - m\right) \cdot 2 \cdot \frac{1}{2d}\\
& = \frac{1}{2} + \frac{1}{2d}\left((d - \alpha - 2) +(|N_x| - m)\right).
\end{align*}
Thus, the condensed Ricci curvature
\begin{align} \label{Eqn:SRGlb}
\Bbbk(x, y) = 2\kappa_{\frac{1}{2}}(x,y) &\geq 2 \left(1 - \left(\frac{1}{2} + \frac{1}{2d}\left((d - \alpha - 2) +(|N_x| - m)\right)\right)\right) \\
&=  \frac{\alpha + 2}{d} - \frac{|N_x| - m}{d}.\notag 
\end{align}

Next we define a Kantorovich potential function $f:V \to {\R}$ with respect to the matching $\mathcal{M}$ by
\begin{equation} \label{Eqn:SRGKantorovichPotential}
f(z) =
\begin{cases}
1 & z = x \\
1 & z \in N_x \setminus A_{N_y}(N_x)\\
-1 & z \in A_{N_y}(N_y)\\  
0 & \text{otherwise}
\end{cases}
\end{equation}
where $A_{N_y}(N_x)$ and $A_{N_y}(N_y)$ denote the vertices in $N_x$ and $N_y$, respectively,  that lie along alternating paths initiated in $N_y$. Taking $T=N_x$ and $S=N_y$ in Lemma $\ref{Lem:Lip}$, it follows that there are no edges between $N_x \setminus A_{N_y}(N_x)$ and $A_{N_y}(N_y)$. Hence, $f$ is a $1$-Lipschitz function on $V$. 

Let $a_x = |A_{N_y}(N_x)|$ and $a_y = |A_{N_y}(N_y)|$. Then 
$$|N_x \setminus A_{N_y}(N_x)| = |N_x| - a_x = d-\alpha-1-a_x$$
and by Lemma \ref{Lem:CountingPaths}  
$$a_y =  |A_{N_y}(N_y)| = |A_{N_y}(N_x)| + |N_y| - m = a_x + |N_x| - m.$$ 
Therefore, by the Kantorovich duality theorem, we find that the transport distance
\begin{align*}
W(m_x, m_y) & \geq \sum_{z \in V} f(z) (m_x(z) - m_y(z)) = \left(\frac{1}{2} - \frac{1}{2d}\right) + (d -\alpha - 1 - a_x) \cdot  \frac{1}{2d} + a_y \cdot \frac{1}{2d}  \\
& = \left(\frac{1}{2} - \frac{1}{2d}\right) + (d - \alpha - 1 - a_x) \frac{1}{2d} + (a_x + |N_x| - m) \frac{1}{2d} \\
& = \frac{1}{2} + \frac{1}{2d} \left((d - \alpha - 2) + (|N_x| - m)\right).
\end{align*}
Similar to the calculation \eqref{Eqn:SRGlb}, we have
\begin{equation} \label{Eqn:SRGub}
\Bbbk(x,y) = 2\kappa_{\frac{1}{2}}(x,y) \leq \frac{\alpha + 2}{d} - \frac{|N_x| - m}{d}.
\end{equation}
Hence, from the inequalities \eqref{Eqn:SRGlb} and \eqref{Eqn:SRGub}, it follows 
$$\Bbbk(x, y) = \frac{\alpha + 2}{d} - \frac{|N_x| - m}{d}.$$

Now if $\mathcal{M}$ is a perfect matching between $N_x$ and $N_y$, then the size of $\mathcal{M}$ is by definition $m =|N_x|=|N_y| = d - \alpha -1$. Moreover, since all vertices in $N_x$ and $N_y$ are matched by some edge in $\mathcal{M}$, it follows that there are no alternating paths initiated in $N_x$ or $N_y$ in the induced bipartite subgraph between vertices in $N_x$ and $N_y$. But then $A_{N_y}(N_x) = \varnothing$ and $A_{N_y}(N_y) = \varnothing$. Taking all of this into account, we see that the same transport plan \eqref{Eqn:SRGTransportPlan} coupled with the Kantorovich potential \eqref{Eqn:SRGKantorovichPotential} yield the equality
$$\Bbbk(x, y) = \frac{\alpha + 2}{d}$$
in the case of a perfect matching  between $N_x$ and $N_y$.
\end{proof}


\subsection{Strongly Regular Graphs of Girths 4 and 5} \label{Sec:SRGg4g5}
 
In this subsection we exploit the adjacency properties of strongly regular graphs $G=(V,E)$ to derive explicit formulas for the condensed Ricci curvature of strongly regular graphs of girths $4$ and $5$. Let $G=(V,E)$ be a strongly regular graph with parameters $(n,d,\alpha,\beta)$. It is well known \cite{CamLint} that if $G$ has girth $5$, then $\alpha =0$ and $\beta =1$, while a strongly regular graph $G$ with girth $4$ has $\alpha =0$ and $\beta \geq 2$. The following theorem is an immediate consequence of Theorem  \ref{Thm:SRGmatching} and the fact that the girth of a graph is its minimal cycle length.

\begin{theorem}
Let $G = (V, E)$ be strongly regular graph with parameters $(n, d, \alpha, \beta)$. If the girth of $G$ is $5$, then the condensed Ricci curvature
$$\Bbbk(x ,y) = \frac{3}{d} - 1$$ 
for all edges $xy \in E$.
\end{theorem}
            
\begin{proof}
Let $xy \in E$ and consider the core neighborhood decomposition of $G$ as in \eqref{Eqn:CoreNbrhd}. Let $x_0 \in N_x$ and suppose that $x_0y_0 \in E$ for some $y_0 \in N_y$. Then $xx_0y_0yx$ is a cycle of length $4$, contradicting the fact that $G$ has girth $5$. Thus, there are no edges between $N_x$ and $N_y$ so the maximum matching between $N_x$ and $N_y$ has size $m = 0$. Therefore, from Theorem \ref{Thm:SRGmatching} with $\alpha = 0$ and $m =0$, we find that the condensed Ricci curvature
$$\Bbbk(x, y) = \frac{3}{d} - 1.\vspace{-0.85cm}$$
\end{proof}

\begin{remark} The transportation plan \eqref{Eqn:SRGTransportPlan} and Kantorovich potential \eqref{Eqn:SRGKantorovichPotential} can be used to compute this curvature formula directly. Since $G$ is girth $5$, there are no edges between vertices in $N_x$ and $N_y$ so $\mathcal{M} = \varnothing$ and each vertex in $N_x$ can be paired to a vertex in $N_y$ through a $2$-step path. Moreover, if we consider induced bipartite subgraph $H$ between $N_x$ and $N_y$, the collection of alternating paths in $H$ that start in $N_y$ consists solely of trivial paths so $A_{N_y}(N_y) = N_y$ and $A_{N_y}(N_x) = \varnothing$.
\end{remark}

In the girth 4 case we appeal to Hall's theorem, which can be used to give necessary and sufficient conditions for perfect matchings on bipartite graphs. Let $H=(V,E)$ be a finite, bipartite graph with parts $S, T$. For any subset $W \subseteq S$, we define 
$$\Gamma_T(W) = \{ v \in T \ \big| \ vw \in E \ \text{for some} \ w \in W\}$$ 
to be the collection of all vertices $v \in T$ sharing an edge with some vertex in $W$. More simply put, $\Gamma_T(W)$ is the neighbor set of $W$ in $T$.

\begin{theorem}[Hall's Theorem] \label{Thm:HallThm}
Let $H = (V, E)$ be a finite, undirected, bipartite graph with parts $S, T$. Then there is a matching that covers $S$ if and only if 
$$|{\Gamma_T(W)}| \geq |{W}|$$ 
for every subset of vertices $W \subseteq S$.
\end{theorem}

With Hall's theorem in hand we are able to establish a perfect matching between the neighbor sets $N_x$ and $N_y$ for strongly regular graphs with girth 4 as a consequence of the pigeonhole principle.
            
\begin{theorem} \label{Thm:SRGgirth4}
Let $G = (V, E)$ be strongly regular graph with parameters $(n, d, \alpha, \beta)$. If the girth of $G$ is $4$, then the condensed Ricci curvature
$$\Bbbk(x ,y) = \frac{2}{d}$$
for all edges $xy \in E$.
\end{theorem}
            
\begin{proof}
Let $xy \in E $ and consider the core neighborhood decomposition of $G$ as in \eqref{Eqn:CoreNbrhd}. Then since $\alpha = 0$, it follows that $G$ is triangle free. Suppose that $x_i \in N_x$. Then since $x_i$ is nonadjacent to $y$, it follows $x_i$ and $y$ have $\beta \geq 2$ common neighbors. Clearly, one such neighbor is $x$ so $x_i$ has $\beta - 1 \geq 1$ neighbors in $N_y$ since $\nabla_{xy} = \varnothing$. Similarly, each vertex $y_j \in N_y$ has $\beta - 1 \geq 1$ neighbors in $N_x$.

Now consider the induced bipartite subgraph $H=G[N_x \cup N_y]$ consisting of all edges in $E$ between vertices in $N_x$ and $N_y$. For sake of contradiction suppose that Let $X \subseteq N_x$ with $|X| = q > |\Gamma_{N_y}(X)| = k$. Since each vertex $x_i \in X \subseteq N_x$ has  $\beta - 1 \geq 1$ neighbors in $N_y$, it follows that there are $(\beta-1)q$ edges between $X$ and $\Gamma_{N_y}(X)$. Therefore, by the generalized pigeonhole principle there exists some $w \in \Gamma_{N_y}(X)$ with 
$$|\Gamma(w)| \geq \ceil*{\frac{(\beta-1)q} {k}} > \ceil*{\frac{(\beta-1)q}{q}}  = \beta - 1$$
since $q > k$. This contradicts the fact that every vertex in $N_y$ has exactly $(\beta - 1)$ neighbors in $N_x$. Thus, $|X| \leq |\Gamma_{N_y}(X)|$ for all subsets $X \subset N_x$ and therefore since $|N_x| = |N_y|$ it follows that there is a perfect matching in $H$ by Hall's Theorem. Noting that $\alpha =0$, it follows from Theorem \ref{Thm:SRGmatching} with $\alpha = 0$ and $m = |N_x|$ that the condensed Ricci curvature
$$\Bbbk(x, y) = \frac{2}{d}.\vspace{-0.85cm}$$
\end{proof}


\bibliographystyle{plain}
\bibliography{coarsericci}

\begin{thebibliography}{10}

\bibitem{BauerJostLiu}
Frank Bauer, J{\"u}rgen Jost, and Shiping Liu.
\newblock Ollivier-{R}icci curvature and the spectrum of the normalized graph
  {L}aplace operator.
\newblock {\em arXiv preprint arXiv:1105.3803}, 2011.

\bibitem{Berge}
Claude Berge.
\newblock Two theorems in graph theory.
\newblock {\em Proceedings of the National Academy of Sciences},
  43(9):842--844, 1957.

\bibitem{BhatMuk}
Bhaswar~B Bhattacharya and Sumit Mukherjee.
\newblock Exact and asymptotic results on coarse {R}icci curvature of graphs.
\newblock {\em Discrete Mathematics}, 338(1):23--42, 2015.

\bibitem{BCLMP}
David Bourne, David Cushing, Shiping Liu, Florentin M\"unch, and Norbert
  Peyerimhoff.
\newblock Ollivier--{R}icci idleness functions of graphs.
\newblock 32, 2017.

\bibitem{CamLint}
P.~Cameron and J.~Lint.
\newblock {\em Graph Theory, Coding Theory and Block Designs}.
\newblock London Mathematical Society Lecture Note Series. Cambridge University
  Press, 1975.

\bibitem{Chung}
Fan R.~K. Chung.
\newblock {\em Spectral graph theory}, volume~92 of {\em CBMS Regional
  Conference Series in Mathematics}.
\newblock Published for the Conference Board of the Mathematical Sciences,
  Washington, DC; by the American Mathematical Society, Providence, RI, 1997.

\bibitem{bakryemery}
David Cushing, Shiping Liu, and Norbert Peyerimhoff.
\newblock Bakry-{{\'E}}mery curvature functions of graphs.
\newblock {\em Canadian Journal of Mathematics}, 2016.

\bibitem{HS}
A.~J. Hoffman and R.~R. Singleton.
\newblock On {M}oore graphs with diameters {$2$} and {$3$}.
\newblock {\em IBM J. Res. Develop.}, 4:497--504, 1960.

\bibitem{JostLiu}
J{\"u}rgen Jost and Shiping Liu.
\newblock Ollivier's {R}icci curvature, local clustering and
  curvature-dimension inequalities on graphs.
\newblock {\em Discrete \& Computational Geometry}, 51(2):300--322, 2014.

\bibitem{Lich}
Andr\'{e} Lichnerowicz.
\newblock {\em G\'{e}om\'{e}trie des groupes de transformations}.
\newblock Travaux et Recherches Math\'{e}matiques, III. Dunod, Paris, 1958.

\bibitem{LLY1}
Yong Lin, Linyuan Lu, and S.-T. Yau.
\newblock {R}icci curvature of graphs.
\newblock {\em Tohoku Mathematical Journal, Second Series}, 63(4):605--627,
  2011.

\bibitem{LLY2}
Yong Lin, Linyuan Lu, and S.-T. Yau.
\newblock {R}icci-flat graphs with girth at least five.
\newblock {\em Comm. Anal. Geom.}, 22(4):671--687, 2014.

\bibitem{LottVillani}
John Lott and C{\'e}dric Villani.
\newblock {R}icci curvature for metric-measure spaces via optimal transport.
\newblock {\em Annals of Mathematics}, pages 903--991, 2009.

\bibitem{Radek}
Florentin M\"unch and Radoslaw~K. Wojciechowski.
\newblock Ollivier {R}icci curvature for general graph {L}aplacians: {H}eat
  equation, {L}aplacian comparison, non-explosion and diameter bounds.
\newblock {\em arXiv preprint arXiv:1712.00875}, 2017.

\bibitem{Ollivier}
Yann Ollivier.
\newblock {R}icci curvature of {M}arkov chains on metric spaces.
\newblock {\em Journal of Functional Analysis}, 256(3):810--864, 2009.

\bibitem{Ollivier2}
Yann Ollivier.
\newblock A survey of {R}icci curvature for metric spaces and {M}arkov chains.
\newblock {\em Probabilistic approach to geometry}, 57:343--381, 2010.

\bibitem{Smith}
Jonathan~DH Smith.
\newblock {R}icci curvature, circulants, and a matching condition.
\newblock {\em Discrete Mathematics}, 329:88--98, 2014.

\bibitem{Sturm}
Karl-Theodor Sturm.
\newblock On the geometry of metric measure spaces.
\newblock {\em Acta mathematica}, 196(1):65--131, 2006.

\end{thebibliography}


\end{document}